\documentclass[abstracton]{scrartcl}

\usepackage[utf8]{inputenc}
\usepackage[T1]{fontenc}
\usepackage[english]{babel}

\usepackage{paralist}
\usepackage{graphicx}
\usepackage{overpic}
\usepackage{subfigure}

\usepackage{booktabs}
\usepackage{dcolumn}

\usepackage{cite}

\usepackage{amsmath}
\usepackage{amssymb}
\usepackage{amsthm}

\usepackage{hyperref}

\theoremstyle{definition}
\newtheorem{defn}{Definition}

\theoremstyle{plain}
\newtheorem{thm}[defn]{Theorem}
\newtheorem{prop}[defn]{Proposition}

\newtheorem{obs}[defn]{Observation}

\title{On a new conformal functional for simplicial surfaces\thanks{This research was supported by the DFG Collaborative Research Center TRR 109, ``Discretization in Geometry and Dynamics''.}}

\author{Alexander I. Bobenko\footnote{Institut f\"ur Mathematik, Technische Universit\"at Berlin, Stra\ss e des 17. Juni 136, 10623 Berlin, Germany} \and Martin P. Weidner\footnotemark[2]}
\date{}

\begin{document}
	\maketitle
	\begin{abstract}\noindent We introduce a smooth quadratic conformal functional and its weighted version $$W_2=\sum_e \beta^2(e)\quad W_{2,w}=\sum_e (n_i+n_j)\beta^2(e),$$ where $\beta(e)$ is the extrinsic intersection angle of the circumcircles of the triangles of the mesh sharing the edge $e=(ij)$ and $n_i$ is the valence of vertex $i$. Besides minimizing the squared local conformal discrete Willmore energy $W$ this functional also minimizes local differences of the angles $\beta$. We investigate the minimizers of this functionals for simplicial spheres and simplicial surfaces of nontrivial topology. Several remarkable facts are observed. In particular for most of randomly generated simplicial polyhedra the minimizers of $W_2$ and $W_{2,w}$ are inscribed polyhedra. We demonstrate also some applications in geometry processing, for example, a conformal deformation of surfaces to the round sphere. A partial theoretical explanation through quadratic optimization theory of some observed phenomena is presented.    
	\end{abstract}
	\section{Introduction. Discrete conformal Willmore functional}	
The Willmore energy of a surface $S\subset {\mathbb R}^3$ is given as 
$$
\int_S (H^2-K)=1/4\int_S (k_1-k_2)^2,
$$
where $k_1$ and $k_2$ denote the principal curvatures, $H=1/2(k_1+k_2)$ and $K=k_1k_2$ the mean and the Gaussian curvatures respectively. For compact surfaces with fixed boundary a minimizer of the Willmore energy is also a minimizer of total curvature $\int_S(k_1^2+k_2^2)$, which is a standard functional in variationally optimal surface modelling. 

In the last years various discretizations of the Willmore functional and of the corresponding flow were investigated. They are mostly used for surface fairing. 
For surface restoration with smooth boundary condition based on a discrete version of the Willmore energy see \cite{Clarenz}. More recently quadratic curvature energy flows were discretized in \cite{Wardetzky} using a semi-implicit scheme. A two step discretization of the Willmore flow was suggested in \cite{Olischlaeger}.

An important feature of the Willmore energy is its conformal invariance, i.e. invariance under M\"obius transformations. 	A conformally invariant discrete analogue of the Willmore functional for simplicial surfaces was introduced in \cite{WillmoreAlt} and studied in \cite{WillmoreFlow}. Recently there was a big progress in development of conformal geometry processing in general \cite {Crane} and in particular in investigation of discrete conformal curvature flows \cite{CranePS}.  
	
The discrete conformal Willmore energy introduced in \cite{WillmoreAlt} is defined in terms of the intersection angles of the circumcircles of neighboring triangles. 
	\begin{defn}
		\label{def:standard}
		Let $S$ be a simplicial surface in 3-dimensional Euclidean space. Denote by $\mathcal{E}$ and $\mathcal{V}$ its edge set and its vertex set respectively. Let $\beta(e_{ij})$ be the external intersection angle of the circumcircles of the two triangles incident with the edge $e_{ij}\in\mathcal{E}$ as shown in Figure \ref{fig:defWillmore}. Then the \emph{discrete conformal Willmore functional} $W(S)$ of $S$ is defined as 
		\begin{equation}
			\label{eq:defClassic}
			W(S):=\sum_{e_{ij}\in\mathcal{E}}\beta(e_{ij})-\pi|\mathcal{V}|,
		\end{equation}
		where $|\mathcal{V}|$ is the number of vertices.
	\end{defn}		
	
	We call the realization of a polyedron \emph{inscribed}, if all its vertices lie on a round sphere. Note that in general we do not require such a realization to be convex. On the other hand we call a polyhedron \emph{inscribable} or \emph{of inscribable type} if there exists a convex, non-degenerate (i.e. without coinciding vertices) inscribed realization. Recall that for inscribed simplicial polyhedra convexity is equivalent to the Delaunay property of the triangulation.  The functional $W$ has two important properties that justify its name.
	
	\begin{figure}
		\hspace*{\fill}
		\begin{overpic}[tics=5,scale=1.5]{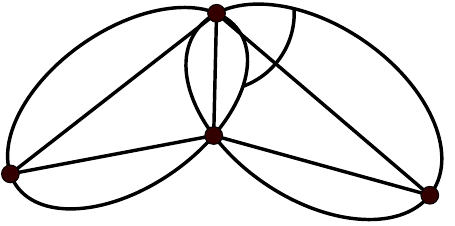}
			\put(65,35){$\beta(e_{ij})$}
			\put(48,35){$e_{ij}$}
			\put(47,50){$v_i$}
			\put(45,12){$v_j$}
		\end{overpic}
		\hspace*{\fill}
		\caption{Definition of the external intersection angle $\beta(e_{ij})$.}
		\label{fig:defWillmore}
	\end{figure}
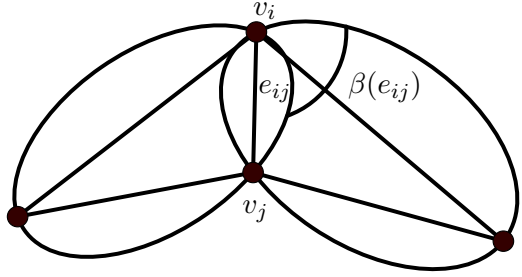
	
	\begin{thm}
		\label{thm:willmore}
		Let $S$ be a simplicial closed surface. Then the following properties hold for the functional $W(S)$.
		\begin{enumerate}[(i)]
			\item $W(S)$ is invariant under conformal transformations of the 3-dimensional Euclidean space (M\"obius transformations).
			\item $W(S)$ is non-negative and it is equal to zero if and only if $S$ is a convex inscribed polyhedron.
		\end{enumerate}
	\end{thm}
	
	The first property follows immediately from the definition since M\"obius transformations preserve circles and their intersection angles. Conformal invariance is an important property of the classical Willmore energy \cite{blaschke,willmoreBook}. The second property is the discrete analogue of the fact that the classic Willmore functional is non-negative and that it is equal to zero if and only if the surface at hand is a (round) sphere. For a proof of (ii) see \cite{WillmoreNeu}. Let us note that the minimizer of $W$ for combinatorial spheres is not unique: $W$ vanishes for any inscribed convex polyhedron, i.e. for any Delaunay triangulation of the round sphere.
	
	The functional $W$ can be used in geometry processing to make the surface ``as round as possible''. In \cite{WillmoreFlow} the associated gradient flow is discussed. It works nicely for smoothing surfaces in many cases. However the functional is not smooth for surfaces that have some of the angles $\beta(e_{ij})$ equal to zero. This happens when the circumcircles of two neighboring triangles coincide. To minimze $W$ numerically it works out quite well to simply set the gradient equal to zero as soon as the angle of the corresponding edge attains a value below a certain threshold \cite{WillmoreFlow}. 
	
	In this paper we introduce a smooth conformal energy for simplicial surfaces, which behaves similar to the discrete Willmore energy \eqref{eq:defClassic}. We have observed several surprising features of the minimizers of this functional. Only very few of them we can explain. The other remain to be challenging problems for future research.	
	
	\section{Quadratic circle-angles functional}
	\label{sec:sol}
	A very natural manner to smoothen $W$ is to consider a quadratic modification of \eqref{eq:defClassic}. 
	
	\begin{defn}
		\label{def:squared}
		Let $\beta(e_{ij})$ be the external intersection angle of the circumcircles as in Definition \ref{def:standard}. Then the \emph{quadratic circle-angles (QCA) functional} $W_2(S)$ is given by
		\begin{align}
			W_2(S)&:=\sum_{e_{ij}\in\mathcal{E}}\beta(e_{ij})^2-c.
		\end{align}
	\end{defn}
	
	The normalization constant $c=4\pi^2\mathbf{1}^t(MM^t)^{-1}\mathbf{1}$ depends only on the combinatorial properties of $S$. Here $M$ is the incidence matrix $M\in\mathbb{R}^{|\mathcal{V}|\times|\mathcal{E}|}$ of the edge graph of the surface and $\mathbf{1}$ is the vector $(1,\ldots,1)^t\in\mathbb{R}^{|\mathcal{V}|}$. This choice of $c$ will be justified in section \ref{sec:theo}. Observe that $W_2$ is smooth at $\beta=0$.
	
	A priori it is not clear for which realization (of a given combinatorics) $W_2$ is minimal. Here an interesting case is the one of inscribable polyhedra because there we can directly compare the result with the minimal realization under the discrete conformal Willmore functional $W$.
		
	Besides $W_2$ we have considered some other modifications among which the most promising is a weighted version of $W_2$.
	\begin{defn}	
		\label{def:weighted}
		Denote by $n_i$ the valence of the vertex $v_i\in\mathcal{V}$. Then the \emph{weighted QCA functional} is given by
		\begin{align*}	
			W_{2,w}(S):=&\sum_{v_i\in\mathcal{V}}\left(\left(\sum_{v_j\sim v_i}\beta(e_{ij})\right)^2+\frac12\sum_{v_k\sim v_i}\sum_{v_j\sim v_i}(\beta(e_{ij})-\beta(e_{ik}))^2\right)-c_w\\
			=&\sum_{e_{ij}\in\mathcal{E}}(n_i+n_j)\beta(e_{ij})^2-c_w.
		\end{align*}
	\end{defn}
	
	The constant $c_w=4\pi^2\mathbf{1}^t(MN^{-1}M^t)^{-1}\mathbf{1}$ again only depends on the combinatorial structure of $S$. Here $M$ is the incidence matrix and $N\in\mathbb{R}^{|\mathcal{E}|\times|\mathcal{E}|}$ is the diagonal matrix with the value $n_i+n_j$ in the row (and column) corresponding to the edge $e_{ij}$. Again the choice of $c_w$ will be motivated in section \ref{sec:theo}. The motivation for the essential part of the functional is the following. For every vertex of the surface, compute the local discrete Willmore functional, square it and add the squares of all angle differences that occur at the given vertex. Hence besides minimizing the squared local discrete Willmore functional, the functional also minimizes local angle differences. A nice feature is that the functional allows a simple formulation using the valences of the vertices. This also shows that $W_{2,w}$ is nothing but a weighted version of $W_2$. In fact $W_2$ and $W_{2,w}$ behave in a similar way, as we shall see in the next section.
	
	\section{Minimization of the QCA functional for various types of discrete surfaces}
	\label{sec:exp}
			All examples have been computed within the VaryLab environment available at \url{http://www.varylab.com} using the limited-memory variable metric (LMVM) method from the TAO project. It only requires the implementation of a gradient, which it uses to compute approximations to the Hessian based on previous iterations. See \cite{Tao} for details. All examples from this article are available as *.obj-files at \url{http://page.math.tu-berlin.de/\~bobenko}. In this section  we only describe the observations made during numerical experiments and the statements are not rigorous. A theoretical analysis is given in the next section.
	\subsection{Inscribable Simplicial Polyhedra}
	Consider a polyhedron of inscribable type. By Theorem \ref{thm:willmore} minimizing $W$ yields a convex inscribed realization. An amazing fact about the minimizers of $W_2$ and $W_{2,w}$ is the following
			
	\begin{obs}
		\label{obs:main}
		For many randomly generated simplicial polyhedra, the minimizers of $W_2$ and $W_{2,w}$ are inscribed polyhedra which are convex in many cases. Moreover these minimizers seem to be unique.
	\end{obs}			
			
	In fact, $W_2$ and $W_{2,w}$ do not only reproduce the qualitative behavior of $W$ in many cases, but they perform better in a certain sense. As an example consider the ellipsoid in Figure \ref{fig:ellipsoid}. It has been obtained by placing 50 vertices randomly on the surface of an ellipsoid and computing their convex hull. The fact that minimizers of the functionals $W_2$ and $W_{2,w}$ are spherical is surprising. The functionals $W_2$ and $W_{2,w}$ yield considerably more uniform triangulations of the sphere, which is not very surprising. Indeed, we have incorporated this feature explicitly into the definition of $W_{2,w}$ by adding the terms that involve the differences angles at incident edges. The functional $W_2$ shows the same behavior since values that are close to each other yield a smaller sum of squares. Then the rate of numerical convergence is faster, i.e. it takes considerably less iterations of the numerical solver to obtain a gradient norm below a certain threshold. For the example in Figure \ref{fig:ellipsoid} this reflects in the following numbers. After 100 minimization steps for $W$, its value is still of order $10^{-2}$. In contrast, computing 100 minimization steps for $W_2$ (resp. $W_{2,w}$) yields a realization where the value of $W$ is of order only $10^{-9}$ (resp. $10^{-10}$). Because of our choice of the normalization constants we have $W_2$ of order $10^{-10}$ and $W_{2,w}$ of order $10^{-8}$ after minimizing the respective energy during 100 steps. We have also considered different initial realisations of the same combinatorial structure. This way, minimizing $W$ can lead to different realizations, all of them satisfying $W=0$. For $W_2$ and $W_{2,w}$ we have always obtained the same realization up to conformal symmetry.
			
	\begin{figure}				
		\hspace*{\fill}		
		\subfigure[]{\includegraphics[width=0.25\textwidth]{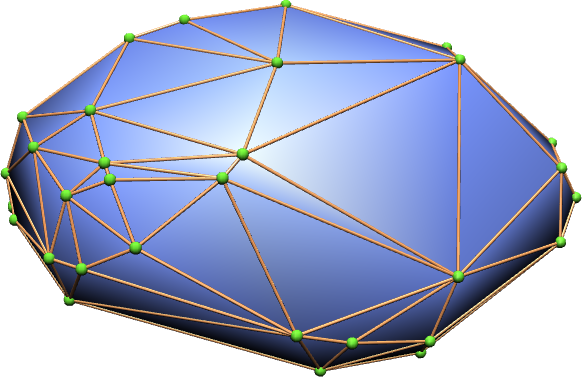}}
		\hspace*{\fill}
		\subfigure[]{\includegraphics[width=0.2\textwidth]{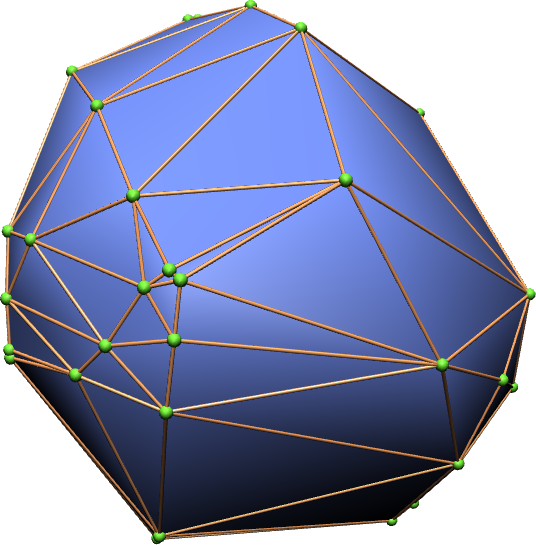}}
		\hspace*{\fill}
		\subfigure[]{\includegraphics[width=.2\textwidth]{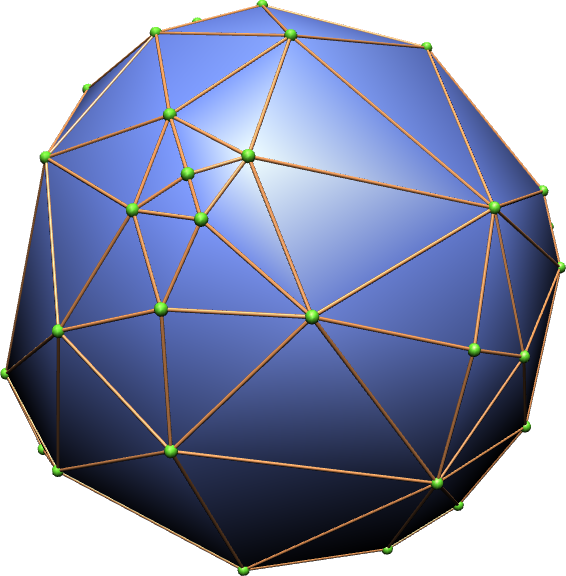}}
		\hspace*{\fill}
		\subfigure[]{\includegraphics[width=.2\textwidth]{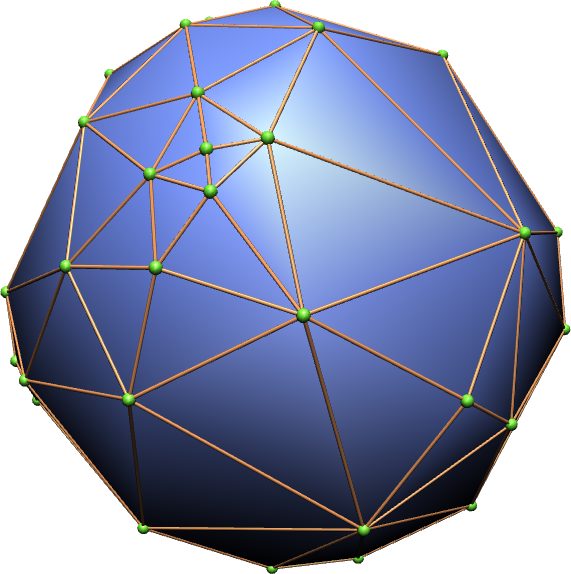}}
		\hspace*{\fill}
		\caption{(a) The original random ellipsoid. (b) The ellipsoid after minimizing $W$. (c) The ellipsoid after minimizing $W_2$. (d) The ellipsoid after minimizing $W_{2,w}$. (b), (c) and (d) are convex inscribed polyhedra.}				
		\label{fig:ellipsoid}
	\end{figure}
			
	The next example is given by the first graph in Figure \ref{fig:graphs}. It is an inscribable polyhedron, i.e. the minimizer for $W$ satisfies $W=0$. For the minimizer of $W_2$ we compute $W_2=0$ but $W>0$. A closer investigation reveals that the minimizer of $W_2$ is a non-Delaunay triangulation of the sphere. In fact, there is one non-Delaunay edge. It is highlighted in the graph by a dotted line. In contrast, the minimizer of $W_{2,w}$ satisfies $W_{2,w}=0$ and also $W=0$, i.e. it is a Delaunay triangulation of the sphere. This is an example where $W_2$ and $W_{2,w}$ yield qualitatively different results. There are also examples for which the minimizer of $W_{2,w}$ is a non-Delaunay triangulation of the sphere. One such example is shown in Figure \ref{fig:graphs}.(b).
			
	\begin{figure}
		\hspace*{\fill}
		\subfigure[]{\includegraphics[width=.3\textwidth]{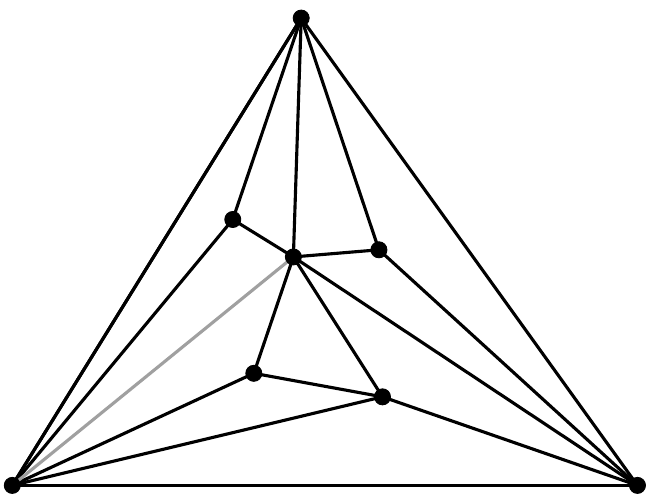}}
		\hspace*{\fill}
		\subfigure[]{\includegraphics[width=.3\textwidth]{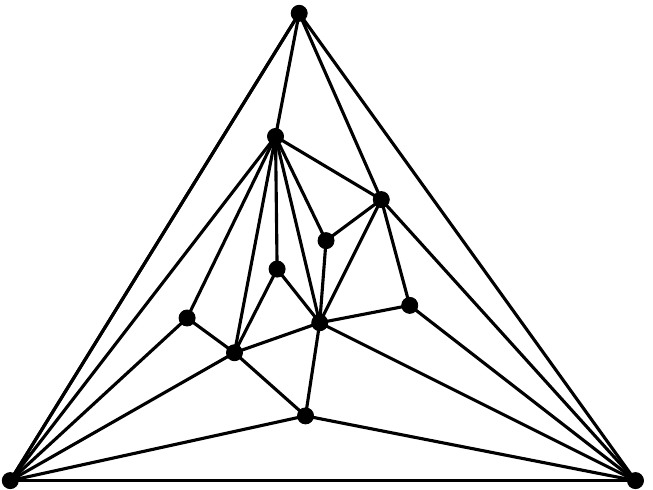}}
		\hspace*{\fill}
		\subfigure[]{\includegraphics[width=.3\textwidth]{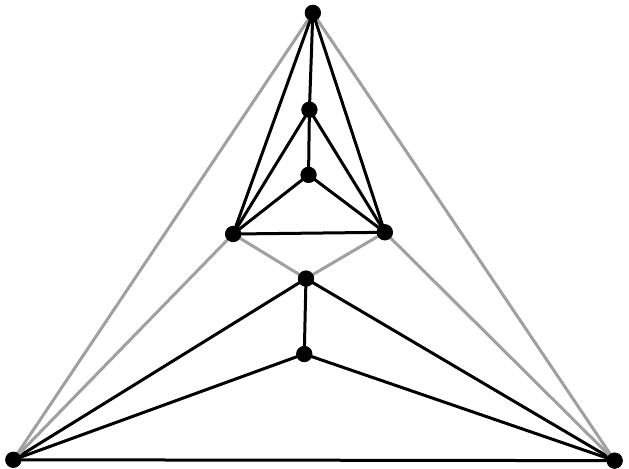}}
		\hspace*{\fill}
		\caption{Three graphs of inscribable type. (a) The graph of a polyhedron for which $W_2$ is minimized by a non-Delaunay triangulation of the sphere. (b) The graph of a polyhedron for which both $W_2$ and $W_{2,w}$ are minimized by a non-Delaunay triangulation of the sphere. (c) The graph of a polyhedron that does not converge while minimizing $W_2$ or $W_{2,w}$.}
		\label{fig:graphs}
	\end{figure}
			
	There are examples that are not covered by Observation \ref{obs:main}. The problem is that there are polyhedra of inscribable type that do not have a realization that minimizes $W_2$ or $W_{2,w}$. Consider the graph in Figure \ref{fig:graphs}.(c). A minimization of $W_2$ or $W_{2,w}$ leads to a realization where several edges collapse. We postpone an explanation of this behavior to the next section.
			
	\subsection{Noninscribable Simplicial Polyhedra}
	In the case of non-inscribable polyhedra, the investigation of the minimizers for $W$, $W_2$ and $W_{2,w}$ is a considerably more difficult task. However, we observe some remarkable phenomena in this case as well.
			
	Consider the example in Figure \ref{fig:selfintersect}. It is not inscribable in a strong sense, but there are convex inscribed realizations with several collapsed edges. Thus if we exclude such degenerate realizations, then $W$ does not have a minimum for this polyhedron. The minimizer for $W_2$ contains a self-intersection but interestingly enough, all its vertices do still lie on a sphere. It is also remarkable that we have $W=2\pi$ for this realization and that the gradient of $W$ vanishes. It is however not a global minimum for $W$. 
			
	\begin{figure}
		\hspace*{\fill}
		\includegraphics[width=.4\textwidth]{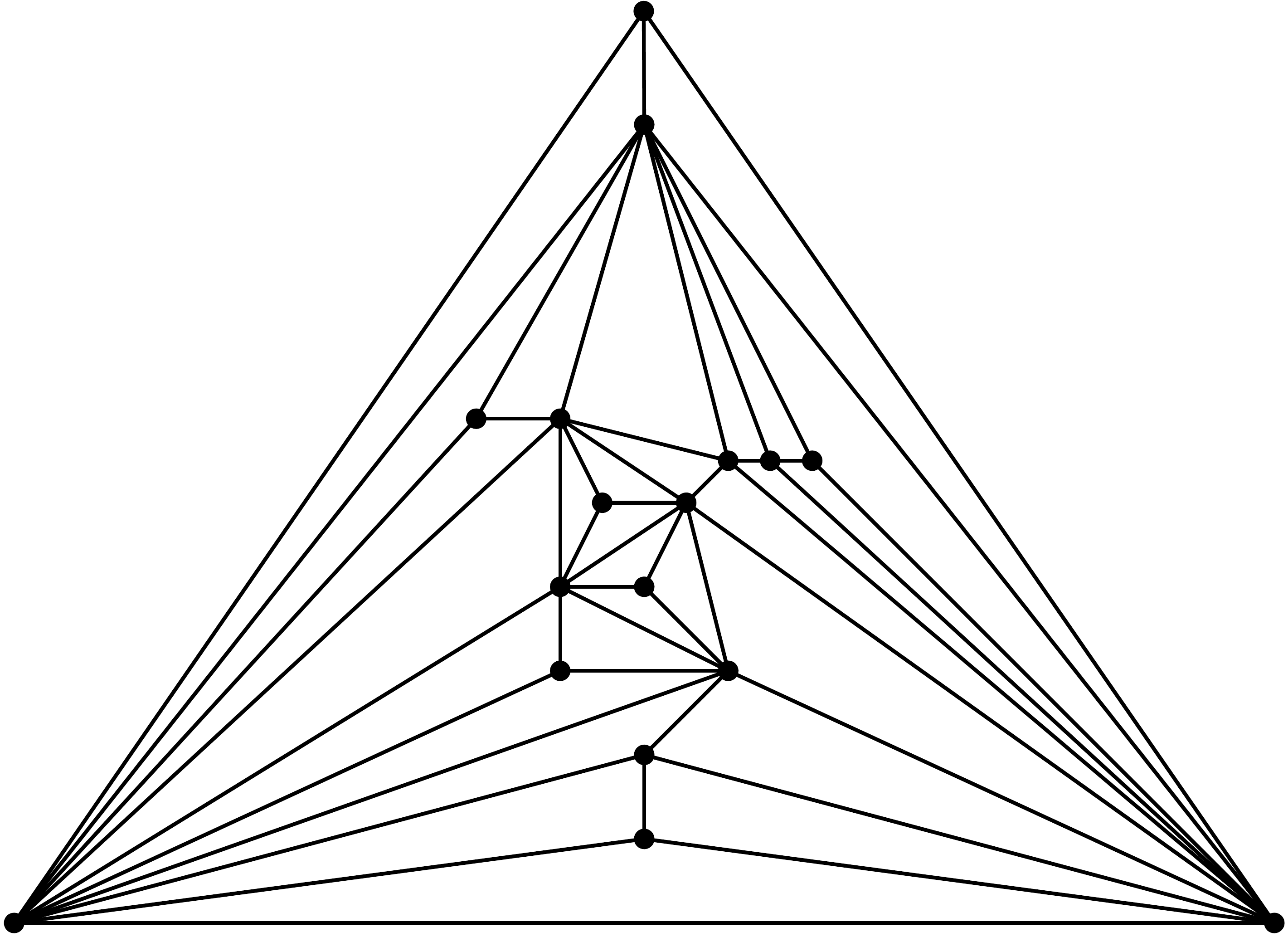}
		\hspace*{\fill}
		\caption{The graph of a polyhedron of non-inscribable type. Its minimizer for $W_2$ contains self-intersections. Minimizing $W_{2,w}$ leads to several collapsed edges.}
		\label{fig:selfintersect}
	\end{figure}			
			
	\subsection{Surfaces of Higher Genus}
	An interesting observation can be made for the minimum of $W_2$ of one particular triangulation of the torus (Figure \ref{fig:torus}). The minimum is attained at the triangulation of a torus of revolution and the ratio of the two radii (measured between appropriate vertices) is equal to $\sqrt{2}$ (up to numerical accuracy). The gradient of $W$ also vanishes for this realization, however this critical point of $W$ is unstable. Starting from the realization in Figure \ref{fig:torus} and minimizing $W$ instead of $W_2$, the numerical solver does not reach the minimal realization. 
			
	Recall the famous Willmore conjecture \cite{willmoreBook} which states that the smooth tori of revolution with a ratio of $\sqrt{2}$ of the two radii (and their M\"obius equivalents) minimize the Willmore energy for tori. The conjecture has recently been proven by Marques and Neves \cite{willmoreProof}.
			
	Computing the value of $W_2$ for the minimal realization gives us $W_2=3.998\pi^2$. By refining the triangulation this value seems to converge to $4\pi^2$. In the smooth case the minimal value of the Willmore energy for tori is equal to $2\pi^2$. 
			
	\begin{figure}
		\hspace*{\fill}
		\subfigure[]{\includegraphics[width=.4\textwidth]{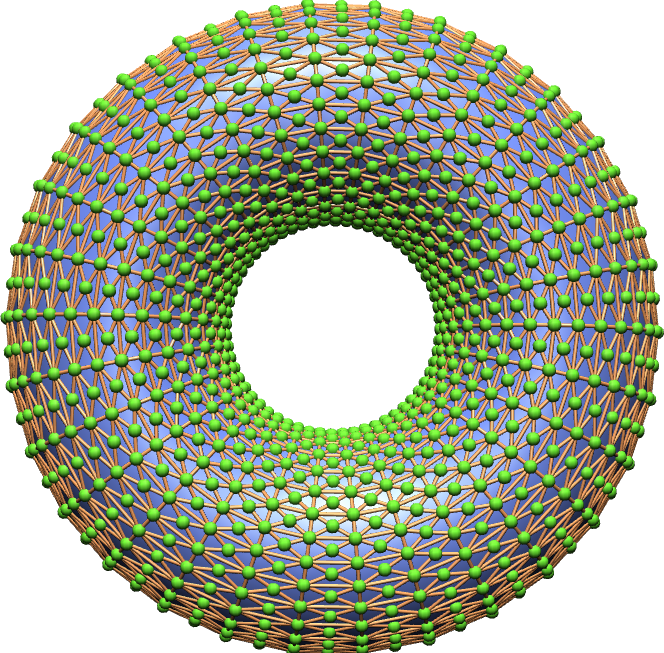}}
		\hspace*{\fill}
		\subfigure[]{\includegraphics[width=.4\textwidth]{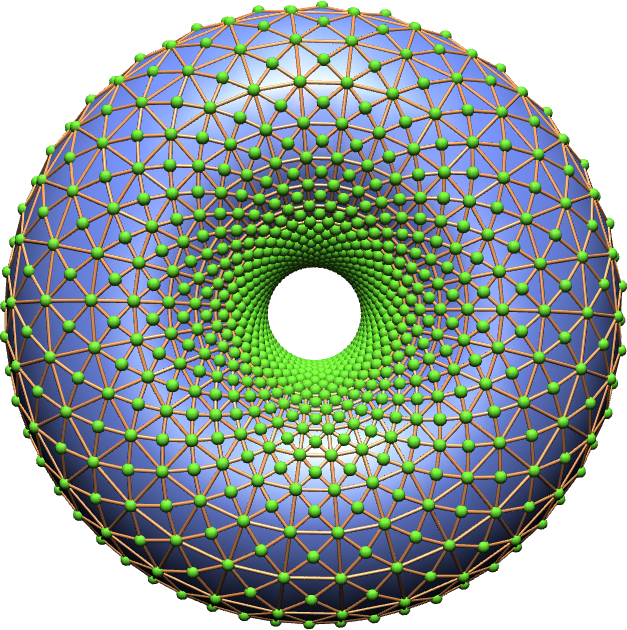}}
		\hspace*{\fill}
		\caption{(a) The original triangulation of a torus of revolution and (b) the result after minimizing $W_2$.}
		\label{fig:torus}
	\end{figure}

	\subsection{Applications in Geometry Processing}
	The Willmore energy functional plays an important role in digital geometry processing and geometric modelling. Applications of the discrete Willmore functional \eqref{eq:defClassic} for non-shrinking surfaces smoothing, surface restoration and hole filling were demonstrated in \cite{WillmoreFlow}. As already mentioned, the main drawback of the functional $W$ is its non-smoothness.
			
	The functionals $W_2$ and $W_{2,w}$ can be applied to the same problems and have some advantages comparing to $W$.
			
	An example is shown in Figure \ref{fig:bunnyhead}. The model is not closed and is treated with fixed boundary conditions, i.e. the boundary curve and tangent planes along it are fixed. The ears of the bunny head cause the solver to run into problems when minimizing $W$. The realization where it gets stuck has many angles $\beta$ with a value smaller than $10^{-3}$ with the smallest angle being even of order $10^{-5}$. Hence the realization is very close to a critical point. In contrast, minimizing $W_2$ makes the bunny head already very spherical after 1000 steps.
			
	The complete bunny shown in Figure \ref{fig:bunny}.(a) is the Stanford bunny in which the holes in the bottom have been filled. Minimizing $W_2$ leads to  a spherical shape with a discrete Willmore energy of $2\pi$. The experiments with the weighted energy $W_{2,w}$ yield even better results. Starting with the model in Figure \ref{fig:bunny}.(a), the surface converges to an inscribed convex realization. After 4000 steps, the value of the discrete Willmore energy is of order $10^{-3}$.
			
	\begin{figure}
		\hspace*{\fill}
		\subfigure[]{\includegraphics[width=.25\textwidth]{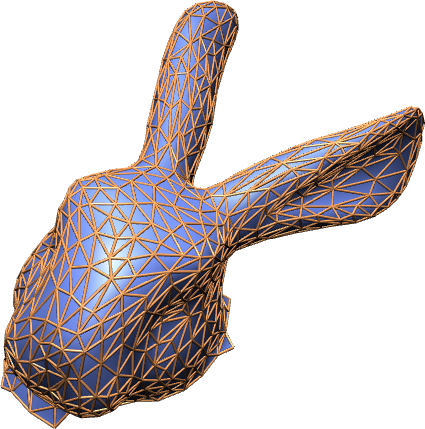}}
		\hspace*{\fill}
		\subfigure[]{\includegraphics[width=.25\textwidth]{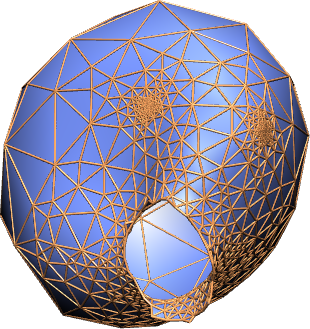}}
		\hspace*{\fill}
		\subfigure[]{\includegraphics[width=.25\textwidth]{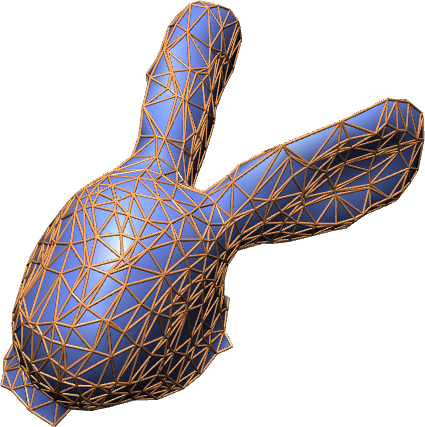}}
		\hspace*{\fill}
		\caption{(a) The original model. (b) The result after 1000 minimization steps for $W_2$. (c) The result after 1000 minimization steps for $W$.}
		\label{fig:bunnyhead}
	\end{figure}		
			
	\begin{figure}
		\hspace*{\fill}
		\subfigure[]{\includegraphics[width=.25\textwidth]{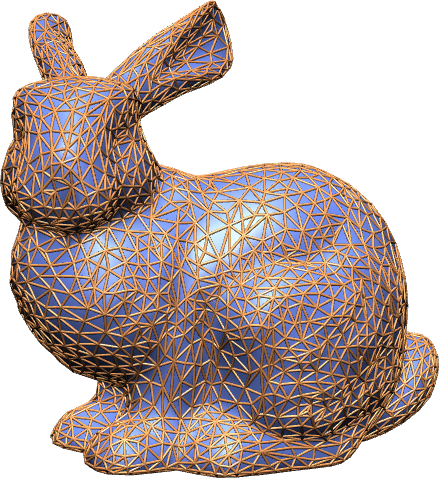}}
		\hspace*{\fill}
		\subfigure[]{\includegraphics[width=.25\textwidth]{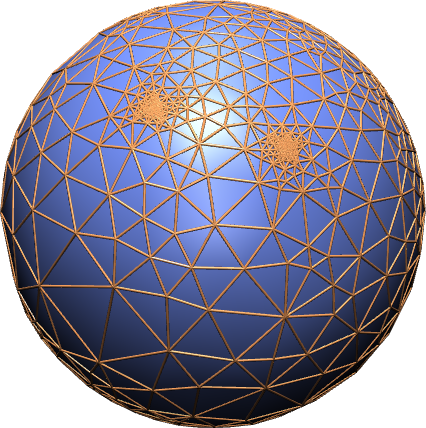}}
		\hspace*{\fill}
		\subfigure[]{\includegraphics[width=.25\textwidth]{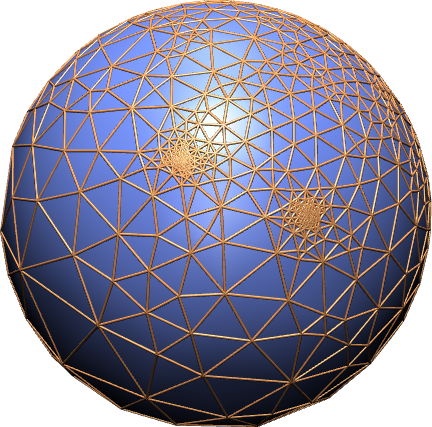}}
		\hspace*{\fill}			
		\caption{(a) The Stanford bunny without holes. (b) The minimizer of $W_2$ after 4000 steps. (c) The minimizer of $W_{2,w}$ after 4000 steps.}
		\label{fig:bunny}
	\end{figure}				
			
	\section{QCA functional and Quadratic Optimization}
	\label{sec:theo}
	For $W$ the minimizers of inscribable polyhedra are convex inscribed realizations. It would be nice to characterize the minimal realizations of these polytopes under $W_2$. In particular it would be interesting to know in which case they are minimizers of $W$, i.e. are convex and inscribed. To investigate the problem, we consider a quadratic program corresponding to $W_2$. At the end of the section we consider also $W_{2,w}$ where similar arguments can be applied.
		
	Suppose we are given the graph $G$ of a simplicial polyhedron of inscribable type. Denote by $\mathcal{V}$ and $\mathcal{E}$ its vertex set and edge set respectively. Now we ignore the geometry and simply consider the intersection angles as arbitrary weights on the edges. The inscribable polyhedra were characterized in \cite{RivinIdeal}.
	\begin{thm}
		\label{thm:rivin}
		Let $\mathcal{P}$ be a convex polyhedron with vertex set $\mathcal{V}$ and edge set $\mathcal{E}$. Let $\beta$ be a weighting of the edges with $0<\beta(e_{ij})<\pi$ for all edges $e_{ij}\in \mathcal{E}$. Then there exists a convex inscribed realization of $\mathcal{P}$ with intersection angles of the circumcircles $\beta$ if and only if the following conditions are satisfied.  
		\begin{compactenum}[(i)]
			\item $\displaystyle \sum_{e_{ij}\sim v_i}\beta(e_{ij})=2\pi$ for every $v_i\in V$. The sum runs over all edges incident with $v_i$. 
			\item $\displaystyle \sum_k\beta(e_k)>2\pi$ for all cycles $e_1^*,\dots,e_n^*$ in the graph of the dual polyhedron that do not bound a face, where $e_k^*$ is the dual edge that corresponds to $e_k$.
		\end{compactenum}
		Moreover, such a realization is unique up to conformal symmetry if it exists.
	\end{thm}
		
	Denote by $M\in\mathbb{R}^{|\mathcal{V}|\times|\mathcal{E}|}$ the incidence matrix of the graph. The set of all $x\in\mathbb{R}^{|\mathcal{E}|}$ that satisfy the constraint that the weights sum up to $2\pi$ around each vertex is then given by solutions of the linear equation $Mx=2\pi\mathbf{1}$ where $\mathbf{1}=(1,\ldots,1)^t\in\mathbb{R}^{|\mathcal{V}|}$. Since we are dealing with the case where $G$ is the graph of a simplicial polyhedron, the matrix $M$ is of full rank $|\mathcal{V}|$ and in particular $MM^t$ is invertible. Thus the following two quadratic programs always have a (unique) solution:
	\begin{align}
		\text{minimize } \Vert x\Vert^2=x^tx \text{ subject to } Mx=2\pi\mathbf{1},
		\label{eq:prog1}\\
		\text{minimize } \Vert x\Vert^2=x^tx \text{ subject to } Mx\geq2\pi\mathbf{1}.
		\label{eq:prog2}
	\end{align}		
					
	By $\Vert\cdot\Vert$ we denote the Euclidean norm. Furthermore, all inequalities between vectors are to be understood component-wise. The angle sum $\sum_{e\sim v}\beta(e)$ for any vertex $v\in\mathcal{V}$ is at least equal to $2\pi$ for every realization of any surface (see \cite{WillmoreNeu}). This means that the solution space of $Mx\geq2\pi\mathbf{1}$ is a superset of all realizable angle sets. In order to find a sufficient condition for the minimum of $W_2$ to be inscribed and convex, we state the following
		
	\begin{prop}
		\label{prop:quadProg}
		Let $x$ and $y$ be the unique solutions of \eqref{eq:prog1} and \eqref{eq:prog2} respectively. Then $x$ and $y$ are equal if and only if the (unique) solution of $MM^t\lambda=2\pi\mathbf{1}$ is non-negative in every component.
	\end{prop}
			
	\begin{proof}
		Suppose that the two minima do not coincide, that is $\Vert  x\Vert>\Vert  y\Vert$. Let $\delta= y- x$ be the difference of the two solutions. Then we have 
		\begin{equation}
			\label{eq:farkas1}
			M\delta=M y-M x\geq2\pi\mathbf{1}-2\pi\mathbf{1}=0.
		\end{equation} 
		Furthermore we know that $\Vert  x\Vert^2>\Vert y\Vert^2$ and hence 		
		\begin{equation*}
			0>\sum_{i=1}^{|E|}\left(( x_i+\delta_i)^2- x_i^2\right)=\sum_{i=1}^{|E|}\left(2 x_i\delta_i+\delta_i^2\right)>2\sum_{i=1}^{|E|} x_i\delta_i.
		\end{equation*}
		Thus we obtain
		\begin{equation}
			\label{eq:farkas2}
			\delta^tx<0.
		\end{equation}
		On the other hand if there is a vector $\delta$ satisfying \eqref{eq:farkas1} and \eqref{eq:farkas2}, we see that $\varepsilon\delta+x$ with some small $\varepsilon>0$ satisfies $M(\varepsilon\delta+x)\geq2\pi\mathbf{1}$ and $\Vert x+\varepsilon\delta\Vert<\Vert x\Vert$. Because of $\Vert y\Vert\leq\Vert x+\varepsilon\delta\Vert$ this implies that $y$ and $x$ cannot coincide. 
			
		Hence the equality $x=y$ is equivalent to the non-existence of $\delta\in\mathbb{R}^{|E|}$ satisfying \eqref{eq:farkas1} and \eqref{eq:farkas2}. By Farkas' lemma (see \cite{Ziegler}), such a $\delta$ exists if and only if there is no $\lambda\geq0$ with $M^t\lambda= x$. Since $M^t$ is injective, $\lambda$ is unique if it exists. It remains to show that it always exists and that it is equal to the unique solution of $MM^t\lambda=2\pi\mathbf{1}$. 
			
		The vector $x$ is the solution of the minimization of $x^tx$ subject to $Mx=2\pi\mathbf{1}$. The respective Lagrange function is given by
			\[ L(x,\tilde\lambda)=x^tx-\tilde\lambda^tMx, \]
		where $\tilde{\lambda}$ is the Lagrange multiplier. The critical point is given by
			\[2x^t-\tilde\lambda^t M=0\Leftrightarrow M^t\tilde\lambda=2x.\]
		Here we see that up to a multiplication by $2$, a solution $\lambda$ of $M^t\lambda=x$ is given by the Lagrange multipliers. Thus the solution always exists and since it is unique, it has to coincide with the solution of $MM^t\lambda=Mx=2\pi\mathbf{1}$.\qed		
	\end{proof}
			
	For any incidence matrix $M$ define 
	\begin{equation}	
		\label{eq:defBeta}		
		\beta(M):=2\pi M^t(MM^t)^{-1}\mathbf{1}.
	\end{equation}
	The matrix $MM^t\in\mathbb{R}^{|\mathcal{V}|\times|\mathcal{V}|}$ is the adjacency matrix of the graph with the valences of the vertices on the diagonal. It is called the signless Laplacian of the graph (see \cite{Laplacian}). The matrix $M^t(MM^t)^{-1}$ is known as the Moore-Penrose pseudoinverse of $M$ (see \cite{inverses}). The proposition shows that in the case $2\pi(MM^t)^{-1}\mathbf{1}\geq0$ it suffices to check whether the vector $\beta(M)$ satisfies  $0<\beta(M)<\pi$ component-wise and condition (ii) from Theorem \ref{thm:rivin}. If this is the case then the minimum of $W_2$ is inscribed and convex and this minimum is unique up to conformal symmetry. We will derive some sufficient conditions for this.
			
	If we assume $2\pi(MM^t)^{-1}\mathbf{1}>0$ instead of $2\pi(MM^t)^{-1}\mathbf{1}\geq0$ then $0<\beta(M)$ is obviously satisfied. Also $\beta(M)<\pi$ holds. Indeed, let us assume that there exists an edge $e$ with $\beta(e)\geq\pi$. Let us denote the corresponding weight by $\beta_e=\beta(e)$. Consider a perturbation of $\beta(M)$ as in Figure \ref{fig:perturb}. Around each vertex the $\beta$'s sum up to $2\pi$ and the perturbation sums up to 0. In particular we have $\beta_a+\beta_b\leq\pi\leq \beta_e$ and $\beta_c+\beta_d\leq\pi\leq \beta_e$ and thus for any $\varepsilon$ satisfying $0<5\varepsilon<\beta_f+\beta_g$,
	\begin{align*}
		&(\beta_e-2\varepsilon)^2+(\beta_a+\varepsilon)^2+(\beta_b+\varepsilon)^2\\
		&+(\beta_c+\varepsilon)^2+(\beta_d+\varepsilon)^2
		+(\beta_f-\varepsilon)^2+(\beta_g-\varepsilon)^2\\
		=&\beta_e^2+\beta_a^2+\beta_b^2+\beta_c^2+\beta_d^2+\beta_f^2+\beta_g^2\\
		&+2\varepsilon(5\varepsilon+\beta_a+\beta_b+\beta_c+\beta_d-2\beta_e-\beta_f-\beta_g)\\
		\leq&\beta_e^2+\beta_a^2+\beta_b^2+\beta_c^2+\beta_d^2+\beta_f^2+\beta_g^2+2\varepsilon(5\varepsilon-\beta_f-\beta_g)\\
		<&\beta_e^2+\beta_a^2+\beta_b^2+\beta_c^2+\beta_d^2+\beta_f^2+\beta_g^2.
	\end{align*}
	This contradicts the minimality of $\beta(M)$ and hence we have $\beta(M)<\pi$.
			
	\begin{figure}	
		\hspace*{\fill}		
		\begin{overpic}[scale=1]{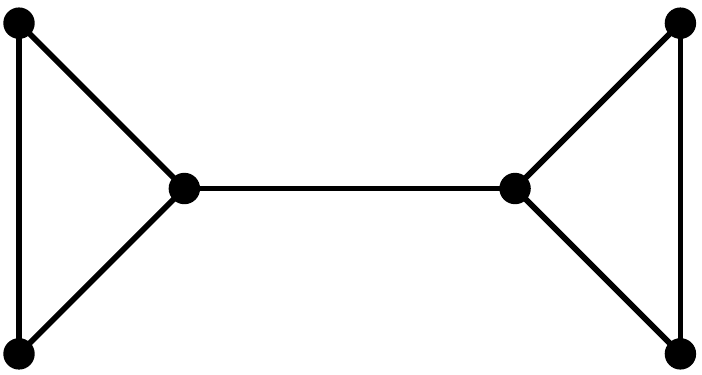}
			\put(42,19){$\beta_e-2\varepsilon$}
			\put(16,10){$\beta_a+\varepsilon$}
			\put(18,37){$\beta_b+\varepsilon$}
			\put(-13,25){$\beta_f-\varepsilon$}
			\put(68,10){$\beta_c+\varepsilon$}
			\put(66,37){$\beta_d+\varepsilon$}
			\put(98,25){$\beta_g-\varepsilon$}
			
		\end{overpic}
		\hspace*{\fill}
		\caption{Perturbing the edge weights on a subgraph.}
		\label{fig:perturb}
	\end{figure}
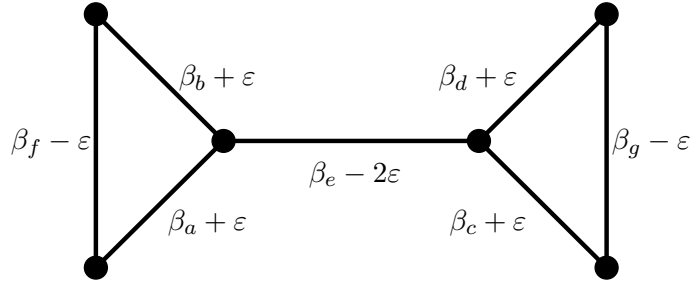
			
	The more complicated question is whether (ii) from Theorem \ref{thm:rivin} is satisfied. The general answer is no as the example in Figure \ref{fig:graphs}.(c) shows. For the angles $\beta(M)$ there is a cocycle (highlighted in the graph) with the angle sum strictly less than $2\pi$. This reflects in the fact that when we minimize $W_2$ numerically several edges collapse.
			
	Let us formulate this claim.
	\begin{prop}
		\label{prop:main}
		Let $\mathcal{P}$ be a polyhedron of inscribable type with incidence matrix $M$. Let $\lambda$ be given by $\lambda=2\pi (MM^t)^{-1}\mathbf{1}$ and let $\beta(M)$ be given by \eqref{eq:defBeta}. Assume that the following two properties are satisfied:
		\begin{compactenum}[(i)]
			\item $\lambda>0$ component-wise
			\item $\beta(M)$ satisfies condition (ii) from Theorem \ref{thm:rivin}.
		\end{compactenum}
		Then the convex inscribed realization given by the angles $\beta(M)$ is a global minimizer of $W_2$. Furthermore, the minimum is unique up to conformal symmetry. 
	\end{prop}
			
	For the angles $\beta(M)$ we have
		\[\beta(M)^t\beta(M)=4\pi^2 \mathbf{1}^t(MM^t)^{-1}\mathbf{1},\] 
	which motivates the choice of the normalization constant in the definition of $W_2$. 
			
	Empirical data suggests that condition (i) is not necessary and can be weakened to $\beta(M)>0$. The problem is that we have no tool to characterize realizable angles as soon as they do not correspond to convex inscribed realizations. 
			
	We have seen that it can happen that $W_2$ is minimized by an inscribed but non-Delaunay realization (Figure \ref{fig:graphs}.(a)). The corresponding angles of the minimizer and the abstract angles given by $\beta(M)$ are shown in Table \ref{tab:nonins}.			
	\begin{table}
		\caption{The angles of the minimizer of $W_2$ obtained numericaly versus the abstract angles $\beta(M)$ given by \eqref{eq:defBeta} for the simplicial surface in Figure \ref{fig:graphs}.(a). All values are divided by $\pi$ and sorted in ascending order.}
		\label{tab:nonins}
		{\scriptsize
		\begin{tabular*}{\linewidth}{@{\extracolsep{\fill}}D{.}{.}{10}D{.}{.}{10}@{}}
			\toprule
			\multicolumn{1}{l}{\normalsize Angles after numerical minimization}  &
			\multicolumn{1}{l}{\normalsize Abstract angles $\beta(M)$ given by \eqref{eq:defBeta}}\\
			\midrule
			0.0295374462	&	-0.0295374466	\\
			0.0559262333	&	0.0559262314	\\
			0.1364420123	&	0.1364420121	\\
			0.1587825189	&	0.1587825174	\\
			0.2392983002	&	0.2392982982	\\
			0.2475447917	&	0.2475447935	\\
			0.3247619735	&	0.3247619762	\\
			0.3504010798	&	0.3504010795	\\
			0.5057350595	&	0.5057350626	\\
			0.5142814330	&	0.5142814304	\\
			0.5163805365	&	0.5163805383	\\
			0.5696079164	&	0.5696079166	\\
			0.6085913486	&	0.6085913487	\\
			0.6724641987	&	0.6724642027	\\
			0.7026013894	&	0.7026013944	\\
			0.7579278849	&	0.7579278807	\\
			0.7831171776	&	0.7831171752	\\
			0.8856735920	&	0.8856735887	\\
			\bottomrule
		\end{tabular*}}
	\end{table}			
	Since the first value in the right-hand column is negative, these values cannot correspond to realizable angles. It is however remarkable that the sign change is the only difference between the two columns (up to numeric accuracy). This phenomenon is still to be clarified.

	Finally we briefly mention how to perform a similar treatment for $W_{2,w}$. The main ingredient is the diagonal matrix $N\in\mathbb{R}^{|\mathcal{E}|\times|\mathcal{E}|}$ that has the value $n_i+n_j$ in the row and column corresponding to the edge $e_{ij}\in\mathcal{E}$. Recall that $n_i$ denotes the valence of the vertex $v_i\in\mathcal{V}$. Thus we now consider the quadratic programs that minimize $x^tNx$ subject to $Mx=2\pi\mathbf{1}$ or $Mx\geq2\pi\mathbf{1}$ respectively. Furthermore, instead of $\beta(M)$ we now consider $\tilde{\beta}(M)$ given by 
	\begin{equation}
		\label{eq:defBetaTilde}
		\tilde{\beta}(M)=2\pi N^{-1}M^t(MN^{-1}M^t)^{-1}\mathbf{1}.
	\end{equation}
	An analog of Proposition \ref{prop:main} then reads as follows.
		
	\begin{prop}
		\label{prop:mainWeighted}
		Let $\mathcal{P}$ be a polyhedron of inscribable type with incidence matrix $M$. Let $\lambda$ be given by $\lambda=2\pi(MN^{-1}M^t)^{-1}\mathbf{1}$ and let $\tilde{\beta}(M)$ be given by \eqref{eq:defBetaTilde}. Assume that the following two conditions are satisfied:
		\begin{compactenum}[(i)]
			\item $\lambda>0$ component-wise
			\item $\tilde{\beta}(M)$ satisfies condition (ii) from Theorem \ref{thm:rivin}.
		\end{compactenum}
		Then the convex inscribed realization given by the angles $\tilde{\beta}(M)$ is a global minimizer of $W_{2,w}$. Furthermore, the minimum is unique up to conformal symmetry. 
	\end{prop}
			
	Again, this motivates the choice of the normalization constant
		 \[\tilde{\beta}(M)^tN\tilde{\beta}(M)=4\pi^2\mathbf{1}^t(MN^{-1}M^t)^{-1}\mathbf{1}\]
	in Defintion \ref{def:weighted}.

	\bibliography{quadraticEnergy}

\begin{thebibliography}{10}

\bibitem{inverses}
A.~Ben-Israel and T.~N.~E. Greville.
\newblock {\em Generalized inverses}.
\newblock Springer-Verlag, New York, second edition, 2003.

\bibitem{blaschke}
W.~Blaschke.
\newblock {\em Vorlesungen über {D}ifferentialgeometrie {III}}.
\newblock Grundlehren der mathematischen Wissenschaften. Springer, 1929.

\bibitem{WillmoreAlt}
A.~I. Bobenko.
\newblock A conformal energy for simplicial surfaces.
\newblock In J.~E. Goodman, J.~Pach, and E.~Welzl, editors, {\em Combinatorial
  and computational geometry}, volume~52 of {\em Math. Sci. Res. Inst. Publ.},
  pages 135--145. Cambridge Univ. Press, Cambridge, 2005.

\bibitem{WillmoreNeu}
A.~I. Bobenko.
\newblock Surfaces from circles.
\newblock In A.~I. Bobenko, P.~Schr{\"o}der, J.~M. Sullivan, and G.~M. Ziegler,
  editors, {\em Discrete differential geometry}, volume~38 of {\em Oberwolfach
  Semin.}, pages 3--35. Birkh\"auser, Basel, 2008.

\bibitem{WillmoreFlow}
A.~I. Bobenko and P.~Schröder.
\newblock Discrete {W}illmore flow.
\newblock In M.~Desbrun and H.~Pottmann, editors, {\em Eurographics Symposium
  on Geometry Processing}, pages 101--110, Vienna, Austria, 2005. Eurographics
  Association.

\bibitem{Clarenz}
U.~Clarenz, U.~Diewald, G.~Dziuk, M.~Rumpf, and R.~Rusu.
\newblock A finite element method for surface restoration with smooth boundary
  conditions.
\newblock {\em Comput. Aided Geom. Design}, 21(5):427--445, 2004.

\bibitem{Crane}
K.~Crane.
\newblock Conformal geometry processing.
\newblock Phd thesis, Caltech, 2013.

\bibitem{CranePS}
K.~Crane, U.~Pinkall, and P.~Schröder.
\newblock Robust fairing via conformal curvature flow.
\newblock {\em ACM Trans. Graph.}, 32(4):61:1--10, 2013.

\bibitem{Laplacian}
D.~Cvetkovi{\'c}, P.~Rowlinson, and S.~K. Simi{\'c}.
\newblock Signless {L}aplacians of finite graphs.
\newblock {\em Linear Algebra Appl.}, 423(1):155--171, 2007.

\bibitem{RivinIdeal}
C.~D. Hodgson, I.~Rivin, and W.~D. Smith.
\newblock A characterization of convex hyperbolic polyhedra and of convex
  polyhedra inscribed in the sphere.
\newblock {\em Bull. Amer. Math. Soc. (N.S.)}, 27(2):246--251, 1992.

\bibitem{willmoreProof}
F.~C. Marques and A.~Neves.
\newblock Min-max theory and the {W}illmore conjecture.
\newblock {\em Ann. Math.}, 179:683--782, 2014.

\bibitem{Tao}
T.~Munson, J.~Sarich, S.~Wild, S.~Benson, and L.~{Curfman McInnes}.
\newblock {TAO} 2.0 users manual.
\newblock Technical Memorandum ANL/MCS-TM-322, Argonne National Laboratory,
  Argonne, Illinois, 2012.

\bibitem{Olischlaeger}
N.~Olischl\"ager and M.~Rumpf.
\newblock Two step time discretization of {W}illmore flow.
\newblock In {\em Mathematics of surfaces {XIII}}, volume 5654 of {\em Lect.
  Note in Comp. Sci.}, pages 278--292. Springer, Berlin, 2009.

\bibitem{Wardetzky}
M.~Wardetzky, M.~Bergou, D.~Harmon, D.~Zorin, and E.~Grinspun.
\newblock Discrete quadratic curvature energies.
\newblock {\em Comput. Aided Geom. Design}, 24(8-9):499--518, 2007.

\bibitem{willmoreBook}
T.~J. Willmore.
\newblock {\em Riemannian geometry}.
\newblock Oxford Science Publications. The Clarendon Press, Oxford University
  Press, New York, 1993.

\bibitem{Ziegler}
G.~M. Ziegler.
\newblock {\em Lectures on polytopes}, volume 152 of {\em Graduate Texts in
  Mathematics}.
\newblock Springer-Verlag, New York, 1995.

\end{thebibliography}
	\bibliographystyle{abbrv}
\end{document}